\documentclass[11pt,oneside]{amsart}
\usepackage{amsfonts, amssymb}
\usepackage{amsmath,amsthm,amscd,xypic}
\usepackage{color}
\usepackage{epic,eepic}
\usepackage{geometry}
\usepackage{graphicx}
\newtheorem{theorem}{Theorem}[section]

\newtheorem*{thm*}{Theorem}
\newtheorem*{mthm*}{Main Theorem}
\newtheorem{lem}[theorem]{Lemma}
\newtheorem{cor}[theorem]{Corollary}
\newtheorem*{cor*}{Corollary}

\newtheorem*{thmA}{Theorem A}
\newtheorem*{thmB}{Theorem B}

\theoremstyle{definition}

\newtheorem{example}[theorem]{Example}

\newtheorem*{conj*}{Conjecture}

\newtheorem{remark}[theorem]{Remark}

\numberwithin{equation}{section}

\newcommand{\red}[1]{\textcolor{red}{#1}}
\newcommand{\snk}{\textsf{sn}_{\kappa}}
\newcommand{\cs}{\textsf{cs}}
\newcommand{\md}{\textsf{md}}
\newcommand{\sn}{\textsf{sn}}

\newcommand{\dist}[1]{{\textsf d}\left(#1\right)}
\newcommand{\diam}{\text{diam}}
\newcommand{\Alex}{\text{Alex\,}}

\newcommand{\Alexnk}{\text{Alex}^n(\kappa)}

\newcommand{\dsp}{\displaystyle}
\newcommand{\geod}[1]{[\,#1\,]}
\newcommand{\geodii}[1]{\,]#1[\,}
\newcommand{\geodci}[1]{[\,#1[\,}
\newcommand{\geodic}[1]{\,]#1\,]}
\newcommand{\ang}[3]{\measuredangle\left[{#1}\,_{#3}^{#2}\right]}
\newcommand{\tang}[4]{\tilde\measuredangle_{#1}\left({#2}\,_{#4}^{#3}\right)}
\newcommand{\cnnt}[2]{{#2}^{*#1}}

\begin{document}




\title{Globalization with probabilistic convexity}
\author{Nan Li}
\address{Department of Mathematics, Penn State University, University Park, PA 16802}
\email{lilinanan@gmail.com, nul12@psu.edu}


\thanks{The author was partially supported by the research funds managed by Penn State University.}


\date{\today}



\maketitle

\pagestyle{headings}


\begin{abstract}
  We introduce a notion of probabilistic convexity and generalize some classical globalization theorems in Alexandrov geometry. A weighted Alexandrov's lemma is developed as a basic tool.
\end{abstract}

\tableofcontents

\section*{Introduction}

Recall that a length metric space $X$ is said to be an Alexandrov space with curvature bounded from below by $\kappa$, if for any quadruple $(p;x_1,x_2,x_3)$, the sum of comparison angles
\begin{align}
  \tang{\kappa}{p}{x_1}{x_2} +\tang{\kappa}{p}{x_2}{x_3}+\tang{\kappa}{p}{x_3}{x_1}\le2\pi,
  \label{comp.e1}
\end{align}
or at least one of the model angles $\tang{\kappa}{p}{x_i}{x_j}$ is not defined. We let $\ang{q}{p}{s}$ denote the angle between two geodesics $\geod{qp}$ and $\geod{qs}$, which is defined by $\limsup\tang{\kappa}qxy$, as $x\in\geod{qp}$, $y\in\geod{qs}$ and $x,y\to q$. The global comparison (\ref{comp.e1}) is equivalent to the Toponogov comparison: for any geodesic $\geod{qs}$ and any points $p\notin\geod{qs}$, plus for $x\in\geod{qs}\setminus\{q,s\}$, we have $\ang{q}{p}{s}\ge\tang{\kappa}{q}{p}{s}$ and $\ang{x}{p}{q}+\ang{x}{p}{s}=\pi$. In the Riemannian  case, the sectional curvature is locally defined and the existence of its lower bound implies the corresponding global Toponogov comparison. It is interesting to consider the similar question in Alexandrov geometry: does local curvature bound imply global curvature bound? This property, if it holds, is so-called globalization property.

For our convenience, we use the following definitions for a local Alexandrov space. An open domain $\Omega$ is called a $\kappa$-domain if for any geodesic $\geod{qs}\subset\Omega$ and any points $p\in\Omega\setminus\geod{qs}$ and $x\in\geod{qs}\setminus\{q,s\}$, we have $\ang{q}{p}{s}\ge\tang{\kappa}{q}{p}{s}$ and $\ang{x}{p}{q}+\ang{x}{p}{s}=\pi$. Clearly, $X$ is an Alexandrov space if and only if it is a $\kappa$-domain. A length metric space $\mathcal U$ is said to be locally curvature bounded from below by $\kappa$ (local Alexandrov space), if for any $p\in \mathcal U$, there is a $\kappa$-domain $\Omega_p\ni p$. Let $\Alex(\kappa)$ and $\Alex_{loc}(\kappa)$ denote the collection of Alexandrov spaces and local Alexandrov spaces with curvature $\ge\kappa$, respectively. Let
$\bar{\mathcal U}$ denote the metric completion of $\mathcal U$, that is, the completion of $\mathcal U$ with respect to its intrinsic metric. The Globalization theorem in \cite{BGP} states that if $\mathcal U\in\Alex_{loc}(\kappa)$ is complete, that is, $\mathcal U=\bar{\mathcal U}$, then ${\mathcal U}\in\Alex(\kappa)$. In general, $\mathcal U$, or $\bar{\mathcal U}$, may not be a global Alexandrov space if\, $\bar{\mathcal U}\neq \mathcal U$ (see Example \ref{example.1}). This is partially because there are points in $\bar{\mathcal U}\setminus \mathcal U$ that are not contained in any $\kappa$-domain. However, this case is particularly important if one wants to prove $X\in\Alex(\kappa)$, but the lower curvature bound can only be verified on a dense subset $\mathcal U\subset X$ (for instance, the manifold points), whose metric completion $\bar{\mathcal U}=X$ (see \cite{GW13} and \cite{HS13}).

For a point $p\in\mathcal U$ and a subset $S\subset\mathcal U$, let
$$\cnnt{p}{S}=\{q\in S: \text{there is a geodesic } \geod{pq} \text{ connecting } p \text{ and } q \text{ in } \mathcal U\}.$$
We rephrase a few classical convexities using the above terminologies.
\begin{itemize}
  \item Convex -- for every point $p\in\mathcal U$, $\cnnt{p}{\mathcal U}=\mathcal U$.
  \item A.e.-convex -- for every point $p\in\mathcal U$, $\mathcal H^n(\mathcal U\setminus\cnnt{p}{\mathcal U})=0$, where $n<\infty$ is the Hausdorff dimension of $\mathcal U$ and $\mathcal H^n$ is the $n$-dimensional Hausdorff measure.
  \item Weakly a.e.-convex -- for any $p\in\mathcal U$ and any $\epsilon>0$, there is $p_1\in B_\epsilon(p)$ so that $\mathcal H^n(\mathcal U\setminus\cnnt{p_1}{\mathcal U})=0$.
  \item Weakly convex -- for any $p, q\in \mathcal U$ and any $\epsilon>0$, there is a point $p_1\in B_\epsilon(p)$ such that $\cnnt{p_1}{B_\epsilon(q)}\neq\varnothing$.
\end{itemize}
In \cite{Pet13}, Petrunin shows that if $\mathcal U\in\Alex_{loc}(\kappa)$ is convex, then $\bar{\mathcal U}\in\Alex(\kappa)$. For example, $\mathcal U$ can be an open convex domain in $\mathbb R^n$. It is proved in \cite{GW13} that the above statement remains true if ``convex" is replaced by ``a.e.-convex". In this case, $\mathcal U$ can be an open convex domain in $\mathbb R^n$ with finitely many points removed.

In this paper, we introduce a notion of probabilistic convexity and prove some globalization theorems related to it. Let $p\in\mathcal U$ and $\geod{qs}$ be a geodesic in $\mathcal U$. Consider the probability that a point on $\geod{qs}$ can be connected to $p$ by a geodesic in $\mathcal U$:
$${\bf Pr}\left(p\prec\geod{qs}\right)=\frac{\mathcal
  H^1\left(\cnnt p{\geod{qs}}\right)} {\mathcal
  H^1\left(\geod{qs}\right)}.
$$
Here $\mathcal H^1$ denotes the $1$-dimensional Hausdorff measure. We say that $\mathcal U$ is weakly $\mathfrak p_\lambda$-convex if for any $p,q,s\in\mathcal U$ and any $\epsilon>0$, there are points $p_1\in B_\epsilon(p)$, $q_1\in B_\epsilon(q)$, $s_1\in B_\epsilon(s)$ and a geodesic $\geod{q_1s_1}\subset\bar{\mathcal U}$ so that ${\bf Pr}(p_1\prec\geod{q_1 s_1})>\lambda-\epsilon$. By taking $s\in B_\epsilon(q)$, we see that if $\lambda>0$, then
weak $\mathfrak p_\lambda$-convexity implies weak convexity. If $\mathcal U\in\Alex_{loc}(\kappa)$, then weak a.e.-convexity implies weak $\mathfrak p_1$-convexity (see the proof of Corollary \ref{cor.ThmA}). Our main results are stated as follows.

\begin{thmA}
  If $\mathcal U\in\Alex_{loc}(\kappa)$ is weakly $\mathfrak p_1$-convex, then its metric completion $\bar{\mathcal U}\in\Alex(\kappa)$.
\end{thmA}

The following theorem is about the optimal lower curvature bound for the metric completion.

\begin{thmB}
  Suppose that $\mathcal U\in\Alex_{loc}(\kappa)$ is weakly $\mathfrak p_\lambda$-convex for some $\lambda>0$. If $\bar{\mathcal U}\in\Alex(\kappa_0)$ for some $\kappa_0\in\mathbb R$, then $\bar{\mathcal U}\in\Alex(\kappa)$.
\end{thmB}

With some extra argument, we prove the following corollary.

\begin{cor}\label{cor.ThmA}
If $\mathcal U\in\Alex_{loc}(\kappa)$ is weakly a.e.-convex then its metric completion $\bar{\mathcal U}\in\Alex(\kappa)$.
\end{cor}

These results provide some connections between the global geometry and the local geometry via a probability of convexity, which may be used to attack a long-standing conjecture mentioned in M. Gromov¡¯s book \cite{Gromov-SaGMoC}: 

\begin{conj*} If $X\in\Alexnk$ has no boundary, then
a convex hypersurface in $X$ equipped with the intrinsic metric is an Alexandrov
space with the same lower curvature bound.
\end{conj*}

 The answers to the following two questions remain unclear to the author.
 \begin{itemize}
   \item{\it Is Theorem A true or false if $\mathcal U$ is weakly $\mathfrak p_\lambda$-convex for some $\lambda\in(0,1)$?}
   \item {\it Is Theorem B true or false if $\mathfrak p_\lambda$-convexity is replaced by weak convexity?}
 \end{itemize}
A negative answer to the second question may provide some new examples for Alexandrov spaces.


%

The main issue to adapt the proof from the complete case $\bar{\mathcal U}=\mathcal U$ to the incomplete case $\bar{\mathcal U}\neq\mathcal U$ is that there is not a priori uniform lower bound for the size of $\kappa$-domains in a bounded subset of $\mathcal U$. Our proof is divided into four steps. In Section 1 we establish the key tool ``weighted Alexandrov's lemma". In Section 2 we recall and prove some comparison results for thin triangles near a geodesic which can be covered by $\kappa$-domains. In Section 3, the probability condition and the
weighted Alexandrov's lemma are used to prove a global comparison through a combination of thin triangles. We complete the proof in Section 4 and give some examples in Section 5.




%
%
%

The author would like to thank Dmitri Burago and Anton Petrunin for helpful discussions.


\bigskip

\begin{center} {\bf Notation and conventions} \end{center}

\begin{itemize}
  \item $\mathbb M_\kappa^n$ -- the $n$-dimensional space form with constant curvature $\kappa$.
      \smallskip
  \item $\dist{p,q}$ or $|pq|$ -- the distance between points $p$ and $q$.
      \smallskip
  \item $\geod{pq}$ -- a minimal geodesic connecting points $p$ and $q$ if it exists. Once it appears, it will always mean the same geodesic in the same context. For simplicity, we let $\geodic{pq}=\geod{pq}\setminus\{p\}$, $\geodci{pq}=\geod{pq}\setminus\{q\}$ and $\geodii{pq}=\geod{pq}\setminus\{p,q\}$.
      \smallskip
  \item $\tang{\kappa}qps$ -- the angle at $q$ of the model triangle $\tilde\triangle_\kappa pqs$, where $\tilde\triangle_\kappa pqs$ is a geodesic triangle in $\mathbb M_\kappa^2$ with the same lengths of sides as $\triangle pqs$.
      \smallskip
  \item $\ang{q}{p}{s}$ -- the angle at $q$ between geodesics $\geod{qp}$ and $\geod{qs}$.
      \smallskip
  \item $\psi(\epsilon\mid\delta)$ -- a positive function in $\epsilon$ and $\delta$ which satisfies $\dsp\lim_{\epsilon\to 0^+}\psi(\epsilon\mid\delta)=0$.
\end{itemize}

\section{Weighted Alexandrov's lemma}

Alexandrov's lemma plays an important role in the classical globalization theorem. It is used to combine two small comparison triangles. In our case, we need to combine multiple triangles whose comparison curvatures are not necessarily the same.  In this section, $X$ is a general length metric space. We assume all triangles are contained in a ball of radius $\frac{\pi}{10\sqrt\kappa}$ in the case $\kappa>0$. Recall the Alexandrov's lemma.

\begin{lem}[\cite{AKP} Alexandrov's lemma]\label{Alex.lem}
  Let $p,q,s$ and $x\in\geodii{qs}$ be points in $X$. Then $\tang{\kappa}{q}{p}{x}\ge\tang{\kappa}{q}{p}{s}$ if and only if $\tang{\kappa}{x}{p}{q}+\tang{\kappa}{x}{p}{s}\le\pi$.
\end{lem}

By a direct computation, we have the following two weighted Alexandrov's lemmas in implicit forms.

 \begin{lem}\label{tri.comp.2.2-im} Let $p,q,s$ and $x\in\geodii{qs}$ be points in $X$. Suppose that $|qx|=b$ and $|xs|=d$. If there are $\kappa_1, \kappa_2\in\mathbb R$ such that \begin{align}
   \tang{\kappa_1}{x}{p}{q}+\tang{\kappa_2}{x}{p}{s}\le\pi,
   \label{tri.comp.2.2-e0}
 \end{align}
  then there is $\lambda \in(0,1)$, 
  such that for $\bar\kappa=(1-\lambda)\kappa_1+\lambda\kappa_2$, we have
  \begin{align} \tang{\kappa_1}{q}{p}{x}\ge\tang{\bar\kappa}{q}{p}{s}.
  \end{align}
 \end{lem}

\begin{lem}\label{triangle.comp.mult.im}
  Let $p,q,s$ and $x_i, i=1,2,\dots,N$ be points in $X$, where $x_0=q$, $x_N=s$ and $x_i\in\geodii{qs}$ appear in the same order. If there are $\kappa_i\in\mathbb R$, such that
  \begin{align}
    \tang{\kappa_i}{x_i}p{x_{i-1}}+\tang{\kappa_{i+1}}{x_i}p{x_{i+1}}\le\pi,
  \end{align}
  $i=1,2,\dots,N-1$, then there is $$\bar\kappa=\min_i\{\kappa_i\} +\psi\left(\max_{i,j}\{|k_i-k_j|\}\right.\left|\; \min_i\{|px_i|\},|qs|\right)$$
  such that
  \begin{align}
    \tang{\kappa_1}{q}p{x_1}\ge\tang{\bar\kappa}{q}p{s}.
  \end{align} 
\end{lem}

When $|qs|$ is small enough, depending on $|pq|$ and $\kappa_i$, Taylor expansions can be used to get an effective weighted Alexandrov's lemma. Let
  $$\snk(t)=\left\{
      \begin{array}{ll}
        \frac{\sin(\sqrt\kappa t)}{\sqrt\kappa}, & \hbox{$\kappa>0$,} \\
        t, & \hbox{$\kappa=0$,} \\
        \frac{\sinh(\sqrt{-\kappa} t)}{\sqrt{-\kappa}}, & \hbox{$\kappa<0$,}
      \end{array}
    \right.
    \quad \quad
    \cs_\kappa(t)=\snk'(t)
    =\left\{
      \begin{array}{ll}
        \cos(\sqrt\kappa t), & \hbox{$\kappa>0$,} \\
        1, & \hbox{$\kappa=0$,} \\
        \cosh(\sqrt{-\kappa} t), & \hbox{$\kappa<0$,}
      \end{array}
    \right.
  $$
  and
  $$\md_\kappa(t)=\int_0^t\snk(a)\,da
    =\left\{
      \begin{array}{ll}
        \frac{1-\cos(\sqrt\kappa t)}{\kappa}, & \hbox{$\kappa>0$,} \\
        \frac{t^2}{2}, & \hbox{$\kappa=0$,} \\
        \frac{1-\cosh(\sqrt{-\kappa} t)}{\kappa}, & \hbox{$\kappa<0$.}
      \end{array}
    \right.
  $$
For any $\kappa\in\mathbb R$ and $\triangle_\kappa ABC\subset\mathbb M_\kappa^2$, the $\kappa$-cosine law can be written in the following form
  \begin{align}
    \md_\kappa(|BC|)
    &=\md_\kappa(|AB|)+\md_\kappa(|AC|)
    -\kappa\cdot\md_\kappa(|AB|)\md_\kappa(|AC|)
    \notag
    \\
    &-\sn_\kappa(|AB|)\sn_\kappa(|AC|)\cos\ang{A}{B}{C}.
    \label{cos.law}
  \end{align}
  Let $|AB|=c$, $|AC|=b$ and $|BC|=a$. The Taylor expansion of $|BC|$ at $b=0$ is
  \begin{align}
    a=c-b\cos(\ang ABC)+\frac12\sin^2(\ang ABC) \frac{\cs_\kappa(c)}{\snk(c)}\cdot b^2+O\left(b^3\right).
    \label{tri.comp.2.e1}
  \end{align}
  The second order term in the expansion gives an indication of the curvature. The coefficient of the
  third term depends on both $c$ and $\kappa$. Let $f_c(\kappa)=\frac{\cs_\kappa(c)}{\snk(c)}$ and view it as a function in $\kappa$. Sometimes we write $f_c(\kappa)$ as $f(\kappa)$ if $c$ is relatively fixed. Direct computation shows that $f(\kappa)$ is a $C^2$, concave and strictly decreasing function for $\kappa\in\left(-\infty,\; \left(\frac{2\pi}{c}\right)^2\right)$.

\begin{lem}\label{tri.comp.2.2}
Let the assumption be the same as in Lemma \ref{tri.comp.2.2-im}. Let  $a=|pq|$ and $\bar\kappa\in\mathbb R$ such that
  \begin{align}
    f_a(\bar\kappa)=\frac{(b^2+2bd)f_a(\kappa_1)+d^2f_a(\kappa_2) }{(b+d)^2}.
    \label{tri.comp.2.2.e0}
  \end{align}
  There is $\delta=\delta(a,\kappa_1,\kappa_2)>0$ so that if $|qs|<\delta$
  then
  \begin{align}
    \tang{\kappa_1}{q}{p}{x}\ge\tang{\bar\kappa}{q}{p}{s}.
    \label{tri.comp.2.2.e00}
  \end{align}
 \end{lem}

 \begin{remark}By the monotonicity and the concavity of $f$, we see that
 $$\bar\kappa
   \ge\frac{(b^2+2bd)\kappa_1+d^2\kappa_2}{(b+d)^2}
   \ge\min\left\{\kappa_1,  \; \frac{b^2\kappa_1+d^2\kappa_2}{b^2+d^2}\right\}.
 $$
 \end{remark}


 \begin{proof}[{\bf Proof of Lemma \ref{tri.comp.2.2}}]
  Not losing generality, assume $\ang qp{s}\neq 0,\pi$. Our goal is to combine $\triangle pqx$ and $\triangle pxs$ and get a globalized comparison. We start from finding a connection between $\tang{\kappa_1}{q}{p}{x}$ and $\tang{\kappa_2}{x}{p}{s}$, namely, (\ref{tri.comp.2.2.e8}) and (\ref{tri.comp.2.2.e9}). Let $|px|=c$. Applying (\ref{tri.comp.2.e1}) on $\triangle_{\kappa_1} \tilde p\tilde q\tilde x\subset \mathbb M_{\kappa_1}^2$, we get
  \begin{align}
    c&= a-b\cos\left(\tang{\kappa_1}{q}{p}{x}\right) +\frac12\sin^2\left(\tang{\kappa_1}{q}{p}{x}\right) f_a(\kappa_1) \cdot b^2+O(b^3)
  \label{tri.comp.2.2.e2}
    \\
    a&= c-b\cos\left(\tang{\kappa_1}{x}{p}{q}\right) +\frac12\sin^2\left(\tang{\kappa_1}{x}{p}{q}\right) f_c(\kappa_1) \cdot b^2+O(b^3)
  \label{tri.comp.2.2.e3}.
  \end{align}
  Adding up (\ref{tri.comp.2.2.e2}) and (\ref{tri.comp.2.2.e3}) and taking in account $|a-c|<b$, we get
  \begin{align}
    &\cos\left(\tang{\kappa_1}{q}{p}{x}\right)
    +\cos\left(\tang{\kappa_1}{x}{p}{q}\right)
    \notag \\
    &\qquad =\frac12\sin^2\left(\tang{\kappa_1}{q}{p}{x}\right) f_a(\kappa_1) \cdot b
    +\frac12\sin^2\left(\tang{\kappa_1}{x}{p}{q}\right) f_c(\kappa_1) \cdot b
    +O(b^2),
    \notag\\
    &\qquad =\frac12\left(\sin^2\left(\tang{\kappa_1}{q}{p}{x}\right)
    +\sin^2\left(\tang{\kappa_1}{x}{p}{q}\right)\right) f_a(\kappa_1) \cdot b
    +O(b^2).
  \label{tri.comp.2.2.e4}
  \end{align}
  Therefore,
  \begin{align}
    &\sin^2\left(\tang{\kappa_1}{x}{p}{q}\right)
    -\sin^2\left(\tang{\kappa_1}{q}{p}{x}\right)
    \notag\\
    &\qquad=\left(\cos\left(\tang{\kappa_1}{q}{p}{x}\right)
    +\cos\left(\tang{\kappa_1}{x}{p}{q}\right)\right) \left(\cos\left(\tang{\kappa_1}{q}{p}{x}\right)
    -\cos\left(\tang{\kappa_1}{x}{p}{q}\right)\right)
    \notag\\
    &\qquad\le 2f_a(\kappa_1) \cdot b
    +O(b^2).
    \label{tri.comp.2.2.e5}
  \end{align}
  Plugging (\ref{tri.comp.2.2.e5}) back into (\ref{tri.comp.2.2.e4}), we get
  \begin{align}
    \cos\left(\tang{\kappa_1}{q}{p}{x}\right)
    +\cos\left(\tang{\kappa_1}{x}{p}{q}\right)
    \le\sin^2\left(\tang{\kappa_1}{q}{p}{x}\right)
    f_a(\kappa_1) \cdot b
    +O(b^2).
  \label{tri.comp.2.2.e6}
  \end{align}

  The assumption
  \begin{align}
    \tang{\kappa_1}{x}{p}{q}+\tang{\kappa_2}{x}{p}{s}\le\pi,
    \label{tri.comp.2.2.e6.1}
  \end{align}
  implies
  \begin{align}
    0\le\cos\left(\tang{\kappa_1}{x}{p}{q}\right) +\cos\left(\tang{\kappa_2}{x}{p}{s}\right).
    \label{tri.comp.2.2.e7}
  \end{align}
  Summing (\ref{tri.comp.2.2.e6}) and (\ref{tri.comp.2.2.e7}):
  \begin{align}
    -\cos\left(\tang{\kappa_2}{x}{p}{s}\right)
    \le -\cos\left(\tang{\kappa_1}{q}{p}{x}\right)
    +\sin^2\left(\tang{\kappa_1}{q}{p}{x}\right)
    f_a(\kappa_1) \cdot b
    +O(b^2).
    \label{tri.comp.2.2.e8}
  \end{align}
  This also implies that
  \begin{align}
    &\sin^2\left(\tang{\kappa_2}{x}{p}{s}\right)
    -\sin^2\left(\tang{\kappa_1}{q}{p}{x}\right)
    \notag\\
    &\qquad\le \left(\cos\left(\tang{\kappa_1}{q}{p}{x}\right)
    -\cos\left(\tang{\kappa_2}{x}{p}{s}\right)\right) \left(\cos\left(\tang{\kappa_1}{q}{p}{x}\right)
    +\cos\left(\tang{\kappa_2}{x}{p}{s}\right)\right)
    \notag\\
    &\qquad\le 2f_a(\kappa_1) \cdot b
    +O(b^2).
    \label{tri.comp.2.2.e9}
  \end{align}
  Apply (\ref{tri.comp.2.e1}) again for $\triangle_{\kappa_2} \tilde p\tilde x\tilde s\subset \mathbb M_{\kappa_2}^2$ and by the property of $f$, \begin{align}
    |ps| &=c-d\cos\left(\tang{\kappa_2}{x}{p}{s}\right) +\frac12\sin^2\left(\tang{\kappa_2}{x}{p}{s}\right) f_c(\kappa_2)\cdot d^2+O(d^3)
    \notag
    \\
    &\le c-d\cos\left(\tang{\kappa_2}{x}{p}{s}\right) +\frac12\sin^2\left(\tang{\kappa_2}{x}{p}{s}\right) f_a(\kappa_2)\cdot d^2+O(d^3).
    \label{tri.comp.2.2.e10}
  \end{align}
  Plug (\ref{tri.comp.2.2.e8}) into (\ref{tri.comp.2.2.e10}),
  \begin{align}
    |ps|
    \le c&-d\cos\left(\tang{\kappa_1}{q}{p}{x}\right)
    +\sin^2\left(\tang{\kappa_1}{q}{p}{x}\right) f_a(\kappa_1) \cdot bd
    \notag\\
    &+\frac12\sin^2\left(\tang{\kappa_1}{q}{p}{x}\right) f_a(\kappa_2) \cdot d^2+O(b^2d)+O(bd^2)+O(d^3).
    \label{tri.comp.2.2.e11}
  \end{align}
  Substitute $c$ by (\ref{tri.comp.2.2.e2}),
  \begin{align}
    |ps|\le a
    &-(b+d)\cos\left(\tang{\kappa_1}{q}{p}{x}\right)
    \notag\\
    &+\frac12\sin^2\left(\tang{\kappa_1}{q}{p}{x}\right)\cdot
    \left(f_a(\kappa_1)\cdot(b^2+2bd)+f_a(\kappa_2)\cdot d^2\right)
    \notag\\
    &+O(b^3)+O(b^2d)+O(bd^2)+O(d^3)
    \notag\\
    \le a&-(b+d)\cos\left(\tang{\kappa_1}{q}{p}{x}\right)
    +\sin^2\left(\tang{\kappa_1}{q}{p}{x}\right)
    f_a(\bar\kappa)\cdot(b+d)^2+O((b+d)^3).
    \label{tri.comp.2.2.e12}
  \end{align}
  where $\bar\kappa$ satisfies $\dsp f_a(\bar\kappa)=\frac{(b^2+2bd)f_a(\kappa_1)+d^2f_a(\kappa_2) }{(b+d)^2}$. 
Apply (\ref{tri.comp.2.e1}) for $\triangle_{\bar\kappa} \tilde p\tilde x\tilde s\subset \mathbb M_{\bar\kappa}^2$:
  \begin{align}
    &|ps|=a-(b+d)\cos\left(\tang{\bar\kappa}{q}{p}{s}\right)
    +\sin^2\left(\tang{\bar\kappa}{q}{p}{s}\right)
    f_a(\bar\kappa)\cdot(b+d)^2
    +O((b+d)^3).
    \label{tri.comp.2.2.e13}
  \end{align}
At last, we compare (\ref{tri.comp.2.2.e13}) with (\ref{tri.comp.2.2.e12}). Note that
  $$-\cos\left(\tang{\bar\kappa}{q}{p}{s}\right)
    +\cos\left(\tang{\kappa_1}{q}{p}{x}\right)
    =(b+d)\left(\tang{\bar\kappa}{q}{p}{s} -\tang{\kappa_1}{q}{p}{x}\right)+O((b+d)^2)
  $$
  and
  $$\sin\left(\tang{\bar\kappa}{q}{p}{s}\right)
    -\sin\left(\tang{\kappa_1}{q}{p}{x}\right)
    =\left(\tang{\bar\kappa}{q}{p}{s}
    -\tang{\kappa_1}{q}{p}{x}\right)+O(b+d).
  $$
  When $|qs|=b+d$ is small, we have
  $$\tang{\kappa_1}{q}{p}{x}\ge\tang{\bar\kappa}{q}{p}{s}.$$
 \end{proof}

We now generalize Lemma \ref{tri.comp.2.2} to the case of multiple triangles.

\begin{lem}\label{triangle.comp.mult} Let the assumption be the same as in Lemma \ref{triangle.comp.mult.im}. Let $c_i=|x_ix_{i+1}|$ and
\begin{align}
    \bar\kappa=\frac{\dsp\sum_{i=1}^N c_i^2\kappa_i
    +2\sum_{i=1}^{N-1} \left(c_{i+1}+\dots+c_N\right)c_i\kappa_i}
    {(c_1+c_2+\dots+c_N)^2}.
    \label{triangle.comp.mult.e1}
  \end{align}
  There is $\delta=\delta(|pq|, \kappa_1\dots,\kappa_N)>0$ so that if $|qs|<\delta$
  then
  \begin{align}
    \tang{\kappa_1}{q}{p}{x_2}\ge\tang{\bar\kappa}{q}{p}{s}.
    \label{triangle.comp.mult.e01}
  \end{align}
\end{lem}

\begin{proof}
  This is proved by a similar comuptation as in Lemma \ref{tri.comp.2.2}. Summing up the cosine laws for the adjacent triangles, we get the following inequality, as a counterpart of (\ref{tri.comp.2.2.e8}):
  \begin{align*}
    -\cos&\left(\tang{\kappa_{i+1}}{x_i}{p}{x_{i+1}}\right)
    \le -\cos\left(\tang{\kappa_1}{x_1}{p}{x_2}\right)
    \\
    &+\sin^2\left(\tang{\kappa_1}{x_1}{p}{x_2}\right)
    [f(\kappa_1)c_1+\dots+f(\kappa_i)c_i]
    +O((c_1^2+\dots+c_i^2)),
  \end{align*}
  A similar argument shows that when $|qs|$ is small, we have
  \begin{align}
    \tang{\kappa_1}{q}p{x_2}\ge\tang{\bar\kappa}{q}p{s},
  \end{align}
  where $\bar\kappa$ satisfies
  \begin{align*}
    f(\bar\kappa)&=
    \frac{\dsp\sum_{i=1}^N c_i^2f(\kappa_i)
    +2\sum_{i=1}^{N-1}\left(f(\kappa_1)c_1+\dots+f(\kappa_i)c_i\right)c_{i+1}} {(c_1+c_2+\dots+c_N)^2}.
  \end{align*}
  By the convexity and monotonicity of $f$,
  \begin{align*}
    \bar\kappa&\ge
    \frac{\dsp\sum_{i=1}^N c_i^2\kappa_i
    +2\sum_{i=1}^{N-1}\left(\kappa_1c_1+\dots+\kappa_ic_i\right)c_{i+1}} {(c_1+c_2+\dots+c_N)^2}
    \\
    &=\frac{\dsp\sum_{i=1}^N c_i^2\kappa_i
    +2\sum_{i=1}^{N-1} \left(c_{i+1}+\dots+c_N\right)c_i\kappa_i} {(c_1+c_2+\dots+c_N)^2}.
  \end{align*}
\end{proof}

In the following we give a special case of Lemma \ref{triangle.comp.mult}, which will be used in our case.

\begin{cor}\label{cor.mult.Alex}
  Let the assumption be as in Lemma \ref{triangle.comp.mult} for $i=1,\dots,2N$. Let $b_i=c_{2i-1}$ and $d_i=c_{2i}$. Assume $\kappa_{2i-1}=\kappa\ge\kappa_{2i}\ge\kappa^*$, $i=1,2,\dots,N$. Then (\ref{triangle.comp.mult.e01}) holds for
  \begin{align}
    \bar\kappa
    =\frac{(b_1+\dots+b_N)^2(\kappa-\kappa^*)} {(b_1+\dots+b_N+d_1+\dots+d_N)^2}+\kappa^*.
    \label{cor.mult.Alex.e1}
  \end{align}
\end{cor}

\begin{proof}
  By the assumption and (\ref{triangle.comp.mult.e1}), the $\bar\kappa$ in Lemma \ref{triangle.comp.mult} satisfies
  \begin{align*}
    \bar\kappa
    &\ge\frac{(b_1+\dots+b_N)^2\kappa +2(b_1+\dots+b_N)(d_1+\dots+d_N)\kappa^*
    +(d_1+\dots+d_N)^2\kappa^*} {(b_1+\dots+b_N+d_1+\dots+d_N)^2}
    \\
    &=\frac{(b_1+\dots+b_N)^2(\kappa-\kappa^*)} {(b_1+\dots+b_N+d_1+\dots+d_N)^2}+\kappa^*.
  \end{align*}
\end{proof}

 In the proof of Theorem A, we also need an estimate for the comparison curvature when a triangle extends without any curvature control.

\begin{lem}\label{tri.comp.2.1}
  Let $x,z,y_0,y\in X$ such that $y_0\in\geodii{xy}_{X}$. Suppose that $|xy_0|=r$ and $|xy|=a$. For any $\kappa\in\mathbb R$, there is $\kappa^*=\kappa^*(a,r,\kappa)\in\mathbb R$ such that $\tang\kappa{x}{y_0}{z}\ge\tang{\kappa^*}{x}{y}{z}$. Moreover, $\kappa^*$ is decreasing in $a$ and $\kappa^*\to\kappa$ as $a\to r$. In particular, when $|xz|<\delta(r,\kappa)$ is small enough, $\kappa^*$ can be chosen explicitly as $\kappa^*=f_a^{-1}(f_r(\kappa))$.
\end{lem}

\begin{remark}
  Here $\kappa^*=f_a^{-1}(f_r(\kappa))\to-\infty$ as $r\to 0^+$ even if $\frac ar\to 0$ in a constant rate.
\end{remark}

\begin{proof}
\newcommand{\snkb}{\text{sn}_{\underline\kappa}}
\newcommand{\snkt}{\text{sn}_{\bar\kappa}}
  The existence of $\kappa^*$ is obvious by direct computation. Suppose that $|xz|=b$ is small. By the Taylor expansion (\ref{tri.comp.2.e1}) and the triangle inequality,
  \begin{align}
    |yz|&\le |y_0y|+|y_0z|
    \notag\\
    &= |y_0y|+|xy_0|-b\cos\left(\tang\kappa{x}{y_0}{z}\right) +\sin^2\left(\tang\kappa{x}{y_0}{z}\right)f_r(\kappa) \cdot b^2+O(b^3)
    \notag\\
    &=a-b\cos\left(\tang\kappa{x}{y_0}{z}\right) +\sin^2\left(\tang\kappa{x}{y_0}{z}\right)f_a(\kappa^*) \cdot b^2+O(b^3).
    \label{tri.comp.2.1.e1}
  \end{align}
  On the other hand,
  \begin{align}
    |yz|= a-b\cos\left(\tang{\kappa^*}{x}{y}{z}\right) +\sin^2\left(\tang{\kappa^*}{x}{y}{z}\right)f_a(\kappa^*) \cdot b^2+O(b^3).
    \label{tri.comp.2.1.e2}
  \end{align}
  Compare (\ref{tri.comp.2.1.e1}) with (\ref{tri.comp.2.1.e2}) in a way similar to Lemma \ref{tri.comp.2.2}. We get that when $b$ is sufficiently small,
  $$\tang{\kappa^*}{x}{y}{z}\le\tang\kappa{x}{y_0}{z}.$$
\end{proof}

\section{Comparisons near $\kappa$-geodesics}

In this section, we always assume that $\mathcal U\in\Alex_{loc}(\kappa)$. By the standard globalization process, we may assume $\diam(\mathcal U)\le\frac{\pi}{2\sqrt\kappa}$ if $\kappa>0$ and $\diam(\mathcal U)\le 1$ if $\kappa\le0$. If the metric completion $\bar{\mathcal U}$ is locally compact, then for any $p,q\in\bar{\mathcal U}$, there exists a geodesic $\geod{pq}$ connecting $p$ and $q$ in $\bar{\mathcal U}$. If $\bar{\mathcal U}$ is not locally compact, we consider its $\omega$-power $\bar{\mathcal U}^\omega$, where $\omega$ is a fixed non-principle ultrafilter on natural numbers (see \cite{AKP} for more details). $\bar{\mathcal U}$ can be viewed as a subspace of $\bar{\mathcal U}^\omega$. For any two points $p,q\in\bar{\mathcal U}$, there exists a geodesic $\geod{pq}$ connecting $p$ and $q$ in $\bar{\mathcal U}^\omega$. In either case, geodesic $\geod{pq}$ is well defined and $\geod{pq}\subset\mathcal U$ means that $p$ and $q$ can be connected by a geodesic in $\mathcal U$.

Let $\gamma$ be a geodesic. We call $\gamma$ a $\kappa$-geodesic if every point on $\gamma$ is contained in a $\kappa$-domain. By the definition, geodesic $\gamma_1$ is also a $\kappa$-geodesic with the same $\kappa$-domain covering, if $\gamma_1$ is close enough to $\gamma$. We start with recalling some results from \cite{Pet13}.
\begin{lem}[Lemma 2.3 in \cite{Pet13}]\label{Pet13-Lem2.3}
  Let $\Omega_p$ and $\Omega_q$ be two $\kappa$-
domains in $\bar{\mathcal U}$. Let $p\in\Omega_p$, $q\in\Omega_q$
and $\geod{pq}\in\Omega_p\cup\Omega_q$.
Then for any geodesic $\geod{qs}\subset \Omega_q$, we have $\ang{q}{p}{s}\ge\tang{\kappa}{q}{p}{s}$ if $|qs|$ is sufficiently small.
\end{lem}

\begin{cor}[Corollary 2.4 in \cite{Pet13}]\label{Pet13-Cor2.4}
Let $\Omega_1$ and $\Omega_2$ be two $\kappa$-domains in $\bar{\mathcal U}$. Assume $\Omega_3\subset\Omega_1\cup\Omega_2$
is an open set such that for any two points $x, y\in\mathcal U$, any geodesic $\geod{xy}$ lies
in $\Omega_1\cup\Omega_2$. Then $\Omega_3$ is a $\kappa$-domain.
\end{cor}

We observe that the proof of Lemma 2.5 in \cite{Pet13} also works for the case $x=p$. Thus the following stronger result holds. For completeness, we repeat Petrunin's proof here.

\begin{lem}\label{Pet13-Lem2.5.2}
  Let $\geod{pq}$ be a $\kappa$-geodesic and the points $x, y$ and $z$ appear
on $\geodci{pq}$ in the same order. Assume that there are $\kappa$-domains $\Omega_1\supset\geod{xy}$ and $\Omega_2\supset\geod{yz}$. Then
\begin{enumerate}
  \item geodesic $\geod{xz}$ is unique;
  \item for any $\epsilon>0$, there is $\delta>0$ such that $\geod{uv}\subset B_{\epsilon}(\geod{xz})$ for any $u\in B_{\delta}(\geod{xy})$ and $v\in B_{\delta}(\geod{yz})$.
\end{enumerate}
In particular, there is an open set
$\Omega_3\subset\bar{\mathcal U}$ which contains $\geod{xz}$ and such that for any two points $u,v\in \Omega_3\cap\mathcal U$, any geodesic $\geod{uv}$ lies in $\Omega_1\cup\Omega_2$. By Corollary \ref{Pet13-Cor2.4}, $\Omega_3$ is a $\kappa$-domain.

\end{lem}

\begin{proof}
  (1) follows from the fact that for any point $z$ in a $\kappa$-domain and any points $p,q,s\in\mathcal U$, $\ang{z}{p}{q}+\ang{z}{p}{s}+\ang{z}{q}{s}\le2\pi$. This is proved exactly the same way as in \cite{BGP}. (2) is argued by contradiction. Assume that there exists $\epsilon>0$ and a sequence of geodesics $\geod{u_iv_i}$, such that $u_i\to u$, $v_i\to v$ but $\geod{u_iv_i}\not\subset B_\epsilon(\geod{xz})$. The ultralimit of $\geod{u_iv_i}$ is a geodesic in $\bar{\mathcal U}^\omega$, connecting $u$ and $v$. Not losing generality, assume $v\neq u$. Then we obtain a bifurcated geodesic at $v\in\geodii{pq}$. For $r>0$ small enough, $v$ is contained in a $\kappa$-domain $B_r(v)\subset\bar{\mathcal U}$. It is straightforward to verify that $B_r(v)$ is also a $\kappa$-domain in $\bar{\mathcal U}^\omega$, a contradiction.

\end{proof}

\begin{cor}\label{2.k-domain}
  Let $\geod{pq}$ be a $\kappa$-geodesic. For any $y\in\geodii{pq}$, there exist $\kappa$-domains $\Omega_1, \Omega_2$ such that  $\Omega_1\supset\geod{py}$ and $\Omega_2\supset\geod{yq}$.
\end{cor}

\begin{proof}
  Let $\{V_i,\, i=1,2,\dots,N\}$ be a $\kappa$-domain covering of $\geod{pq}$. Let $x_1=p$, $x_{N+1}= q$ and $x_i\in V_{i-1}\cap V_i$, $i=2,3,\dots,N$. We may assume that $x_i$ appear on $\geod{pq}$ in the same order and $V_i\supset\geod{x_ix_{i+1}}$. Not losing generality, assume $x_N=y$. By Lemma \ref{Pet13-Lem2.5.2}, there is a $\kappa$-domain $\Omega_1$ containing $\geod{x_1x_3}$. Thus $\{\Omega_1,V_3,V_4,\dots,V_N\}$ forms a $\kappa$-domain covering of $\geod{pq}$. Repeat applying Lemma \ref{Pet13-Lem2.5.2} as the above, we will arrive at a position that $\geod{pq}$ is covered by $\kappa$-domains $\Omega_{N-2}$ and $V_N$. In fact, we have $\geod{px_N}\subset \Omega_{N-2}$ and $\geod{x_Nx_{N+1}}\subset V_N$.
\end{proof}

Combining Lemma \ref{Pet13-Lem2.5.2} and Corollary \ref{2.k-domain}, we get the following result immediately.

\begin{lem}\label{geod.close}
  Let $\geod{pq}$ be a $\kappa$-geodesic and $x\in\geod{pq}$, $z\in\geodii{pq}$. Then
\begin{enumerate}
  \item geodesic $\geod{xz}$ is unique;
  \item for any $r>0$, there is $\epsilon>0$ such that for any $u,v\in B_\epsilon(\geod{xz})$, any geodesic $\geod{uv}$ is a $\kappa$-geodesic contained in $B_r(\geod{xz})$.
\end{enumerate}
\end{lem}

The following thin triangle comparison is contained in the proof of the main theorem in \cite{Pet13}, which follows directly from Lemma \ref{Pet13-Lem2.3} and Corollary \ref{2.k-domain}.

\begin{lem}\label{thin.comp}
  Let $\geod{pq}$ be a $\kappa$-geodesic. Then there exists $r>0$ such that $\ang{q}{p}{s}\ge\tang{\kappa}{q}{p}{s}$ for any $s\in B_r(q)\setminus\{q\}$.
\end{lem}

%

As an application of the above results, we get the following comparison.

\begin{lem}\label{thin.comp.er}
  Let $\geod{pq}$ be a $\kappa$-geodesic. For any $p_0\in\geodic{pq}$, there exists $r>0$, depending only on $|pp_0|$ and the way that $\geod{pq}$ sits in its $\kappa$-domain covering, such that for with
  any $u,v,w\in B_r(\geod{p_0q})$, we have $\ang{u}{w}{v}\ge\tang{\kappa}{u}{w}{v}$.
\end{lem}

\begin{proof}
  By Lemma \ref{geod.close}, there is $\delta>0$ such that for any $x,y\in B_\delta(\geod{p_0,q})$, any geodesic $\geod{xy}$ is a $\kappa$-geodesic. Apply Lemma \ref{geod.close} once more. Take $r>0$ small so that for any $u,v\in B_r(\geod{p_0q})$, any geodesic $\geod{uv}\subset B_\delta(\geod{p_0q})$. Thus for any $x\in\geod{uv}$, geodesic $\geod{wx}$ is a $\kappa$-geodesic. By Lemma \ref{thin.comp}, the function $g(x)=\textsf{md}_\kappa\circ \dist{w,x}$ is a $(1-\kappa g)$-concave function when restricted to $\geod{uv}$. This implies that $\ang{u}{w}{v}\ge\tang{\kappa}{u}{w}{v}$.
\end{proof}

\begin{cor}\label{uniform.b1-1}
  Let $\geod{pq}$ be a $\kappa$-geodesic. For any $\kappa_1<\kappa$, there exists $r>0$ such that $\ang{x}{p}{y}\ge\tang{\kappa_1}{x}{p}{y}$ for any $x,y\in B_r(q)$. In particular, for any geodesic $\geod{qs}\subset B_r(q)$, we have $\ang{q}{p}{s}\ge\tang{\kappa_1}{q}{p}{s}$ and $\ang{s}{p}{q}\ge\tang{\kappa_1}{s}{p}{q}$.
\end{cor}

\begin{proof}
  For $\epsilon>0$ small, take $p_0\in\geod{pq}$ such that $|pp_0|=\epsilon$. By Lemma \ref{thin.comp.er}, there is $r>0$ such that for any $x,y\in B_r(q)$, $\ang{x}{p_0}{y}\ge\tang{\kappa}{x}{p_0}{y}$. By Lemma \ref{tri.comp.2.1}, we get $\tang{\kappa}{x}{p_0}{y}\ge \tang{\kappa^*}{x}{p}{y}$, where $\kappa^*\to\kappa$ as $p_0\to p$. For any $\kappa_1<\kappa$, $\epsilon$ can be chosen small so that $\kappa^*\ge\kappa_1$. Then we have
  \begin{align*}
    \ang{x}{p}{y}=\ang{x}{p_0}{y}\ge\tang{\kappa}{x}{p_0}{y}
    \ge \tang{\kappa^*}{x}{p}{y} \ge\tang{\kappa_1}{x}{p}{y}.
  \end{align*}
\end{proof}

\begin{lem}\label{exist.ks}
  Let $\geod{qs}$ be a $\kappa$-geodesic and point $p\notin\geod{qs}$. There are $r_0>0$ and $\kappa^*=\kappa^*(\geod{qs},\, \sup\{\dist{p,x}:x\in\geod{qs}\})\in\mathbb R$ such that for any $u\in B_{r_0}(\geod{qs})$ and $v\in B_{r_0}(u)$, we have $\ang{u}{p}{v}\ge\tang{\kappa^*}{u}{p}{v}$.
\end{lem}

\begin{proof}
  By Corollary \ref{2.k-domain}, there exists $r_0>0$ so that for any $u\in B_{r_0}(\geod{qs})$, $B_{4r_0}(u)$ is contained a $\kappa$-domain. Let $w\in\geod{pu}$ such that $|uw|=r_0$. Then we have
  $\ang{u}{w}{v}\ge\tang{\kappa}{u}{w}{v}$. Let $R=\sup\{\dist{p,x}: x\in\geod{qs}\}+10r_0$. By Lemma \ref{tri.comp.2.1}, there is $\kappa^*=\kappa^*(R,r_0,\kappa)$ such that
  $\tang{\kappa}{u}{w}{v}\ge\tang{\kappa^*}{u}{p}{v}$. Therefore, $\ang{u}{p}{v}=\ang{u}{w}{v} \ge\tang{\kappa^*}{u}{p}{v}$. We would like to point out that when $|uv|<\delta(r,\kappa)$ is small, one can select
  $\kappa^*=f_{R}^{-1}(f_{r}(\kappa))$.
\end{proof}


\section{Globalization with weak $\mathfrak p_\lambda$-convexity}

In this section, we always assume $\mathcal U\in\Alex_{loc}(\kappa)$. We first need a better perturbation for weak $\mathfrak p_\lambda$-convexity.

\begin{lem}\label{conv.geod.pert}
  Suppose that $\mathcal U\in\Alex_{loc}(\kappa)$ is weakly $\mathfrak p_\lambda$-convex. If $\lambda>0$, then for any $p,q,s\in\mathcal U$ and any $0<\epsilon<\lambda$, there is a geodesic triangle $\triangle\bar p\bar q\bar s$ such that
\begin{enumerate}
  \item $\bar p\in B_\epsilon(p)$, $\bar q\in B_\epsilon(q)$ and $\bar s\in B_\epsilon(s)$;
  \item $\geod{\bar p\bar q}$ and $\geod{\bar p\bar s}$ are $\kappa$-geodesics;
  \item ${\bf Pr}\left(\bar p\prec\geod{\bar q\bar s}\right)\ge\lambda-\epsilon$.
\end{enumerate}
\end{lem}

\begin{proof}
  Let $0<\epsilon_3\ll\epsilon_2\ll\epsilon_1<\epsilon/10$ be all small. We first select a $\kappa$-geodesic $\geod{p_1q_1}$ that satisfies $p_1\in B_{\epsilon_1}(p)$ and $q_1\in B_{\epsilon_1}(q)$. Let $\bar p_1\in\geod{\bar p_1q_1}$ such that $|p_1\bar p_1|=\epsilon_2$. There is a $\kappa$-geodesic $\geod{p_2s_2}$ that satisfies $p_2\in B_{\epsilon_2}(\bar p_1)$ and $s_2\in B_{\epsilon_2}(s)$. Let $\bar p_2\in\geod{\bar p_2s_2}$ such that $|p_2\bar p_2|=\epsilon_2$. By the definition of weak $\mathfrak p_\lambda$-convexity, there are points $\bar p\in B_{\epsilon_3}(\bar p_2)$, $\bar q\in B_{\epsilon_3}(q_1)$, $\bar s\in B_{\epsilon_3}(s_2)$ and a geodesic $\geod{\bar q\bar s}$, such that ${\bf Pr}\left(\bar p\prec\geod{\bar q\bar s}\right)>\lambda-\epsilon_3>\lambda-\epsilon$.
  By Lemma \ref{geod.close}, when $\epsilon_3\ll\epsilon_2\ll\epsilon_1$, $\geod{\bar p\bar q}$ and $\geod{\bar p\bar s}$ are both $\kappa$-geodesics. At last, we also have
  \begin{align*}
    |q\bar q|
    &\le |qq_1|+|q_1\bar q|\le \epsilon_1+\epsilon_3<\epsilon,
  \end{align*}
  \begin{align*}
    |s\bar s|
    &\le |ss_2|+|s_2\bar s|\le \epsilon_2+\epsilon_3<\epsilon
  \end{align*}
  and
  \begin{align*}
    |p\bar p|
    &\le |pp_1|+|p_1\bar p_1|+|\bar p_1p_2|+|p_2\bar p_2|+|\bar p_2\bar p|
    \\
    &\le 2\epsilon_1+2\epsilon_2+\epsilon_3<\epsilon.
  \end{align*}

\end{proof}



%

The size $r$ of the comparison in Lemma \ref{thin.comp} depends on the way that $\geod{pq}$ is contained in the $\kappa$-domains. When the $r$ is made larger, the comparison curvature may drop. The following two lemmas show that the defection of the lower curvature bound is controlled by the probability of points in $\geod{qs}$ that can be connected to $p$ by a geodesic in $\mathcal U$.

\begin{lem}\label{comp.X1} Let $p\in\mathcal U$, $\kappa,\kappa^*\in\mathbb R$ and $l>0$. For any $\kappa_1<\kappa$, there is $\delta=\delta(|pq|,\kappa_1,\kappa^*)>0$ such that the following holds for any $\kappa$-geodesic $\geod{qs}$ that satisfies $|qs|<\delta$ and $|pq|\ge l$. Suppose $\ang{y_1}{p}{q}\ge\tang{\kappa_1}{y_1}{p}{q}$ for some $y_1\in\geodic{qs}$ and $\ang{x}{p}{y}\ge\tang{\kappa^*}{x}{p}{y}$ for any $x,y\in\geod{qs}$. Then
$$\tang{\kappa_1}{q}{p}{y_1} \ge\tang{\bar\kappa}{q}{p}{s},$$
for $\bar\kappa=\lambda^2\kappa_1+(1-\lambda^2)\kappa^*$, where $\lambda={\bf Pr}(p\prec\geod{qs})$. In particular, if $\geod{pq}$ is a $\kappa$-geodesic, then we have $\ang{q}{p}{s} \ge\tang{\bar\kappa}{q}{p}{s}$.
\end{lem}

\begin{proof}
  Let $\cnnt p{\geod{qs}}$ be the set of points in $\geod{qs}$ which can be connected to $p$ by a geodesic in $\mathcal U$.
  For every $x\in \cnnt p{\geod{qs}}$, $x\neq s$, by Corollary \ref{uniform.b1-1}, there exists $y\in\geodic{xs}$ such that $\ang{x}{p}{y}\ge\tang{\kappa_1}{x}{p}{y}$ and  $\ang{y}{p}{x}\ge\tang{\kappa_1}{y}{p}{x}$. The collection of all such segments $\geod{xy}$, together with $\geod{qy_1}$, gives a covering of $\cnnt p{\geodci{qs}}$. Let $x_1=q$. For any $\eta>0$ small, by Vitali covering theorem, there is a disjoint finite sub-collection $\dsp\{\geod{x_iy_i}\}_{i=1}^N$ so that
  \begin{align}
    \sum_{i=1}^N|x_iy_i|\ge\mathcal H^1\left(\cnnt p{\geod{qs}}\right)-\eta|qs|
    =(\lambda-\eta)|qs|.
    \label{comp.X.e4}
  \end{align}


    \begin{center}\includegraphics[scale=1]{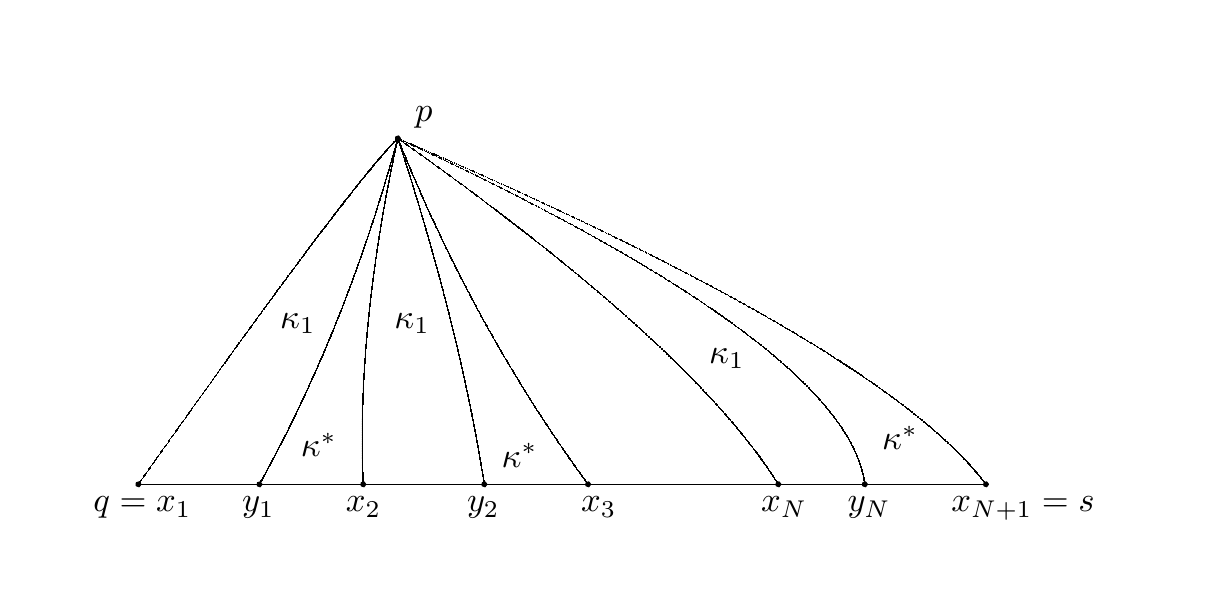}\end{center}

  Let $x_{N+1}=s$. By the construction, for all $i=2,3,\dots, N$, we have
  \begin{align}
    \tang{\kappa_1}{x_i}{p}{y_i}\le\ang{x_i}{p}{y_i}
    \text{\quad and \quad}
    \tang{\kappa_1}{y_i}{p}{x_i}\le\ang{y_i}{p}{x_i}.
    \label{comp.X.e5}
  \end{align}
By the assumptions, we have
  \begin{align}
    \ang{x_i}{p}{y_{i-1}}
    \ge\tang{\kappa^*}{x_i}{p}{y_{i-1}}
    \label{comp.X.e6}
  \end{align}
  and
  \begin{align}
    \ang{y_i}{p}{x_{i+1}}
    \ge\tang{\kappa^*}{y_i}{p}{x_{i+1}}.
    \label{comp.X.e7}
  \end{align}
  Since $\tang{\kappa_1}{y_1}{p}{q}\le\ang{y_1}{p}{q}$, by (\ref{comp.X.e5}), (\ref{comp.X.e6}) and (\ref{comp.X.e7}), together with the fact that $\geod{qs}$ is a $\kappa$-geodesic, we have
  \begin{align}
    \tang{\kappa_1}{x_i}{p}{y_i}+\tang{\kappa^*}{x_i}{p}{y_{i-1}}
    \le \ang{x_i}{p}{y_i}+\ang{x_i}{p}{y_{i-1}}\le\pi
  \end{align}
  and
  \begin{align}
    \tang{\kappa_1}{y_i}{p}{x_i}+\tang{\kappa^*}{y_i}{p}{x_{i+1}}
    \le \ang{y_i}{p}{x_i}+\ang{y_i}{p}{x_{i+1}}\le\pi.
  \end{align}

  Let $b_i=|x_iy_i|$ and $d_i=|y_ix_{i+1}|$, $i=1,2,\dots,N$. By (\ref{comp.X.e4}), we have $\dsp\sum_{i=1}^Nb_i\ge(\lambda-\eta)\sum_{i=1}^N(b_i+d_i)$. Thus when $|qs|$ is small, by Corollary \ref{cor.mult.Alex}, we get
  $$\tang{\kappa_1}{q}{p}{y_1} \ge\tang{\underline\kappa}{q}{p}{s}
  $$
  for
  \begin{align*}
    \underline\kappa
    &=\frac{(b_1+\dots+b_N)^2(\kappa_1-\kappa^*)} {(b_1+\dots+b_N+d_1+\dots+d_N)^2}+\kappa^*
    \\
    &\ge (\lambda-\eta)^2(\kappa_1-\kappa^*)+\kappa^*.
  \end{align*}
  Let $\eta\to 0$. We get the desired result.
\end{proof}

\begin{lem}\label{comp.X} Suppose that $\mathcal U$ is weakly $\mathfrak p_\lambda$-convex.
  Let both $\geod{pq}$ and $\geod{qs}$ be $\kappa$-geodesics. There is $\kappa^*=\kappa^*(p,\geod{qs})\in\mathbb R$, so that for any $\underline\kappa<\lambda^2\kappa+(1-\lambda^2)\kappa^*$,
  \begin{align}
    \ang{q}{p}{s} \ge\tang{\underline\kappa}{q}{p}{s}.
    \label{comp.X.e0}
  \end{align}
  If $\bar{\mathcal U}\in\Alex(\kappa_0)$, then (\ref{comp.X.e0}) holds for any $\underline\kappa<\lambda^2\kappa+(1-\lambda^2)\kappa_0$.
\end{lem}

\begin{proof}
  Let $r_0>0$ and $\kappa^*=\kappa^*(p, \geod{qs})$ be defined as in Lemma \ref{exist.ks}. If $\bar{\mathcal U}\in\Alex(\kappa_0)$, we take $\kappa^*=\kappa_0$. Fix $\epsilon_0>0$ small and take $\bar p_0\in\geod{pq}$ such that $|p\bar p_0|=\epsilon_0$. Let $\delta
  =\delta(\dist{p,\geod{qs}} -2\epsilon_0,\kappa,\kappa^*)\in(0,r_0)$ be determined as in Lemma \ref{comp.X1}.
  By Lemma \ref{thin.comp.er}, there exists $r\in(0,\delta/4)$ such that for any $p_1,x,y\in B_r(\geod{\bar p_0q})$, we have $\ang{x}{p_1}{y}\ge\tang{\kappa}{x}{p_1}{y}$.

  Let $q_i, i=1,\dots N+1$, be a partition of $\geod{qs}$, such that $q_1=q$, $|q_iq_{i+1}|=\delta/2$ and $\delta/2\le|q_{N+1}s|\le\delta$. Let $v_0=q_1$ and $p_0=p$ and $0<\epsilon_i\ll\epsilon_{i-1}\ll r$ be small. By the definition of weak $\mathfrak p_\lambda$-convexity and Lemma \ref{conv.geod.pert}, we can recursively select $p_i, \bar p_i, u_i, v_i$ for $i=1,2,\dots,N$, such that
  \begin{enumerate}
    \item $p_i\in B_{\epsilon_i}(\bar p_{i-1})$,  $u_i\in B_{\epsilon_i}(v_{i-1})$ and $v_i\in B_{\epsilon_i}(q_{i+1})$;
    \item $\geod{p_iu_i}$ and $\geod{p_iv_i}$ are all $\kappa$-geodesics and ${\bf Pr}(p_i\prec\geod{u_iv_i}) \ge \lambda-\epsilon_i$;
    \item $\bar p_i\in\geod{p_iv_i}$ and $|p_i\bar p_i|=\epsilon_i$;
  \end{enumerate}
  \begin{center}\includegraphics[scale=1]{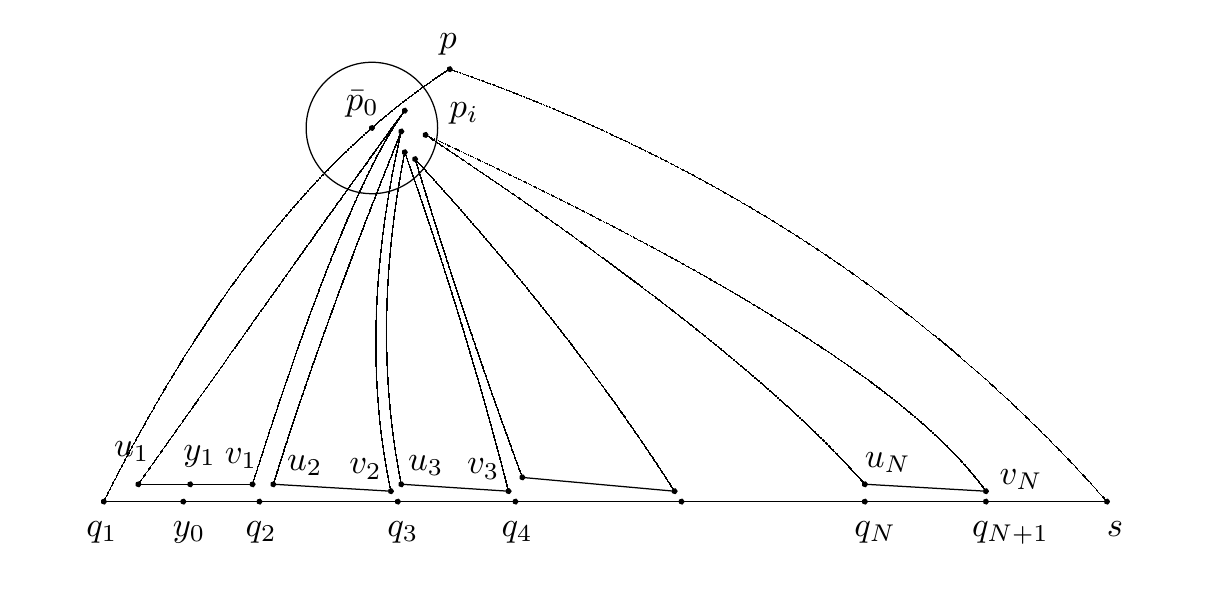}\end{center}

  Let $\epsilon=\sum_{i=1}^N\epsilon_i\ll\epsilon_0$. For a fixed $i$, when $\epsilon_i\to 0$, $\geod{p_iu_i}$ converge to geodesic $\geod{\bar p_{i-1}v_{i-1}}$. Thus when $\epsilon\to 0$, passing to a subsequence,
  $\geod{p_iu_i}$ and $\geod{p_{i-1}v_{i-1}}$  converge to the same limit geodesic $\geod{\bar p_0q_i}$. Note that $\geod{u_iv_i}$ converge to the geodesic $\geod{q_iq_{i+1}}\subset\geod{qs}$ (In the case that $\bar{\mathcal U}$ is not locally compact, we consider the $\omega$-power $\bar{\mathcal U}^\omega$ and the same argument applies). We have
  \begin{align}
    \ang{v_i}{p_i}{u_i}\to\ang{q_{i+1}}{\bar p_0}{q_i}
    \text{\quad and \quad} \ang{u_{i+1}}{p_i}{v_{i+1}}\to\ang{q_{i+1}}{\bar p_0}{q_{i+2}},
  \end{align}
  as $\epsilon\to 0$, $i=1,2,\dots,N-1$, since $q_i$, $i\ge 2$ are interior points of a $\kappa$-geodesic.
  Therefore,
  \begin{align}
    \ang{v_i}{p_i}{u_i}+\ang{u_{i+1}}{p_i}{v_{i+1}}
    \le\ang{q_{i+1}}{\bar p_0}{q_i} +\ang{q_{i+1}}{\bar p_0}{q_{i+2}}+\psi(\epsilon)
    =\pi+\psi(\epsilon).
    \label{comp.X.e3}
  \end{align}
  Let $y_0\in\geod{q_1q_2}$ such that $|q_1y_0|=r/2$. By Lemma \ref{geod.close} and the fact that $|u_1v_1|\ge \delta/2-2\epsilon_1>r$, there exists $y_1\in\geod{u_1v_1}$ such that $d(y_1,y_0)<\psi(\epsilon_1)$. Then $|q_1y_1|\le |q_1y_0|+|y_0y_1|\le\psi(\epsilon_1)+r/2<r$. Thus we have $\ang{q_1}{\bar p_0}{y_0}\ge\tang{\kappa}{q_1}{\bar p_0}{y_0}$
  and $\ang{y_1}{p_1}{u_1}\ge\tang{\kappa}{y_1}{p_1}{u_1}$.
  Let $\kappa_1<\kappa$ such that $\underline\kappa=\lambda^2\kappa_1+(1-\lambda^2)\kappa^*$. The assumptions in Lemma \ref{comp.X1} for $p_i$ and $\geod{u_iv_i}$ are satisfied by the following construction:
  \begin{itemize}
    \item $|u_iv_i|\le|q_iq_{i+1}|+4\epsilon_i <\delta$ and $|p_iu_i|\ge\dist{p,\geod{qs}} -10\epsilon>\dist{p,\geod{qs}} -\epsilon_0$.
    \item $\geod{p_iu_i}$ is a $\kappa$-geodesic.
    \item $|u_iq_i|\le 2\epsilon_i<r_0$ and $|u_iv_i|<\delta<r_0$.
  \end{itemize}
  Thus we have
  \begin{align}
    \ang{u_i}{p_i}{v_i} \ge\tang{\underline\kappa}{u_i}{p_i}{v_i},
    \text{\quad \quad}
    \ang{v_i}{p_i}{u_i} \ge\tang{\underline\kappa}{v_i}{p_i}{u_i},
    \label{comp.X.e1}
  \end{align}
  and
  \begin{align}
    \tang{\kappa_1}{u_1}{p_1}{y_1} \ge\tang{\underline\kappa}{u_1}{p_1}{y_1},
    \label{comp.X.e2}
  \end{align}
  where  $\underline\kappa=(\lambda-\epsilon)^2\kappa_1 +(1-(\lambda-\epsilon)^2)\kappa^*$. By (\ref{comp.X.e3}) and (\ref{comp.X.e1}), we have
  \begin{align}
    \tang{\underline\kappa}{v_i}{p_i}{u_i} +\tang{\underline\kappa}{u_{i+1}}{p_i}{v_{i+1}}
    \le
    \ang{v_i}{p_i}{u_i}+\ang{u_{i+1}}{p_i}{v_{i+1}} \le\pi+\psi(\epsilon).
  \end{align}
  By (\ref{comp.X.e2}) and Alexandrov's lemma, we get
  \begin{align}
    \tang{\kappa_1}{u_1}{p_1}{y_1}
    &\ge\tang{\underline\kappa}{u_1}{p_1}{v_1}
    \notag
    \\
    &\ge\tang{\underline\kappa}{u_1}{p_1}{v_N}-N\psi(\epsilon)
    \ge\tang{\underline\kappa}{u_1}{p_1}{v_N}-|qs|\frac{\psi(\epsilon)}{\delta}.
  \end{align}
  We choose $\epsilon$ small so that $\epsilon, \psi(\epsilon)\ll r,\delta$. Then
  \begin{align}
    \left|\tang{\kappa_1}{q}{\bar p_0}{y_0}-\tang{\kappa_1}{u_1}{p_1}{y_1}\right|<\psi(\epsilon\mid r,\delta).
  \end{align}
  Thus
  \begin{align*}
    \ang{q}{p}{s}
    &=\ang{q}{\bar p_0}{y_0}\ge\tang{\kappa_1}{q}{p_0}{y_0}
    \ge\tang{\kappa_1}{u_1}{p_1}{y_1} -\psi(\epsilon\mid r,\delta)
    \\
    &\ge \tang{\underline\kappa}{u_1}{p_1}{v_N}-\psi(\epsilon\mid r,\delta)
    \ge \tang{\underline\kappa}{q}{p}{s}-\psi(\epsilon\mid r,\delta)-\psi(\epsilon_0)-\psi(\delta).
  \end{align*}
  At last, let $\epsilon\ll\delta\to 0$ and $\epsilon_0\to 0$. We get the desired comparison.
\end{proof}

\begin{cor}\label{comp.X.cor1}
  Suppose that $\mathcal U$ is weakly $\mathfrak p_1$-convex. Let $(p;x_1,x_2,x_3)$ be a quadruple in $\mathcal U$. If $\geod{px_i}$, $i=1,2,3$, are all $\kappa$-geodesics and the comparison angles $\tang{\kappa}{p}{x_i}{x_j}$ are all defined, then
  \begin{align}
       \tang{\underline\kappa}{p}{x_1}{x_2} +\tang{\underline\kappa}{p}{x_2}{x_3} +\tang{\underline\kappa}{p}{x_3}{x_1}\le2\pi.
       \label{comp.X.cor.e1}
  \end{align}
\end{cor}

\begin{cor}\label{comp.X.cor2}
  Suppose that $\mathcal U$ is weakly $\mathfrak p_\lambda$-convex and $\bar{\mathcal U}\in\Alex(\kappa_0)$. Let $(p;x_1,x_2,x_3)$ be a quadruple in $\mathcal U$. If $\geod{px_i}$, $i=1,2,3$, are all $\kappa$-geodesics and the comparison angles $\tang{\kappa}{p}{x_i}{x_j}$ are all defined, then
  \begin{align}
       \tang{\underline\kappa}{p}{x_1}{x_2} +\tang{\underline\kappa}{p}{x_2}{x_3} +\tang{\underline\kappa}{p}{x_3}{x_1}\le2\pi.
       \label{comp.X.cor.e1}
  \end{align}
\end{cor}

\section{Globalization in the metric completion}

The following lemma, together with Corollary \ref{comp.X.cor1} conclude the proof of Theorem A.

\begin{lem}\label{pert.triangle}
  Suppose that $\mathcal U$ is weakly convex and locally curvature bounded from below by $\kappa$. For any $p, x_i\in \bar{\mathcal U}$, $i=1,2,\dots,N$ and any $\epsilon>0$, there are points $\bar p\in B_\epsilon(p)$ and $\bar x_i\in B_\epsilon(x_i)$, $i=1,2,\dots, N$ so that geodesics $\geod{\bar p\bar x_i}$ are all $\kappa$-geodesics.
\end{lem}

\begin{proof}
  The idea has been used in the proof of Lemma \ref{conv.geod.pert}. Let $\epsilon>\epsilon_1\gg\epsilon_2\gg\dots\gg\epsilon_N>0$ be small which will be determined later. First choose $p_1\in B_{\epsilon_1}(p)$ and $\bar x_1\in B_\epsilon(x_1)$ so that $\geod{p_1\bar x_1}_{X}\subset \mathcal U$. Take $\bar p_1\in\geod{p_1\bar x_1}$ with $|\bar p_1 p_1|=\epsilon_1$. Select $p_2\in B_{\epsilon_2}(\bar p_1)$ and $\bar x_2\in B_\epsilon(x_2)$ so that $\geod{p_2\bar x_2}_{X}\subset \mathcal U$. By Lemma \ref{geod.close}, take $0<\epsilon_2\ll\epsilon_1$ so that $\geod{y\bar x_1}$ are $\kappa$-geodesics for all $y\in B_{10\epsilon_2}(\bar p_1)$. In particular, $\geod{p_2\bar x_1}_{X}$ is a $\kappa$-geodesic.

  \begin{center}\includegraphics[scale=1]{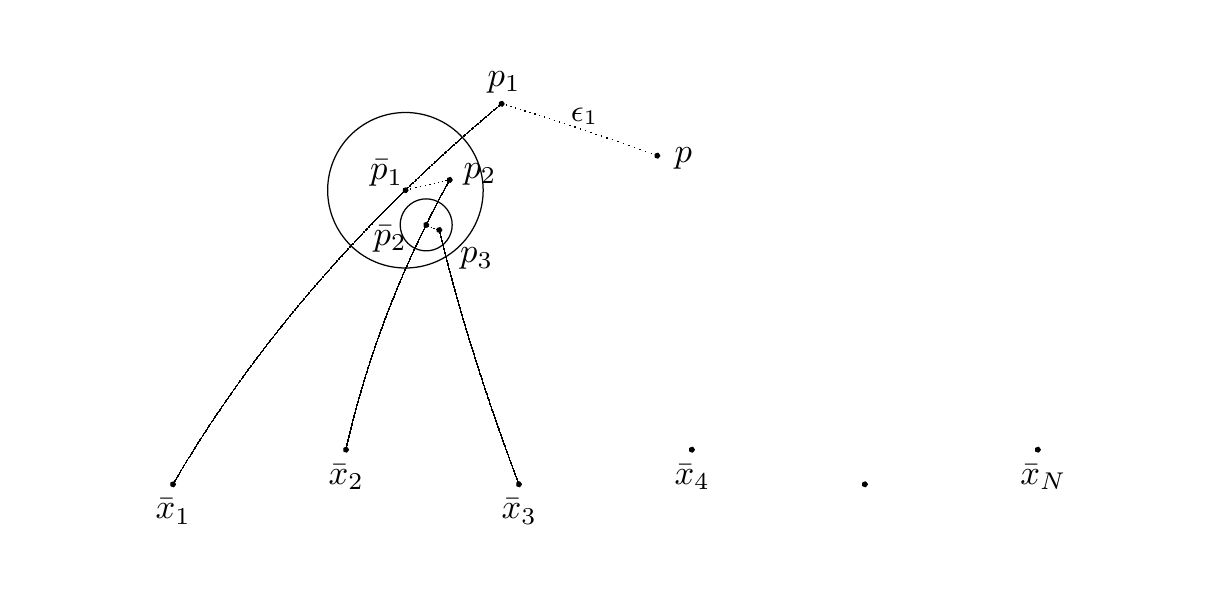}\end{center}

  \vskip -0.2in

  Take $\bar p_2\in\geod{p_2\bar x_2}$ with $|\bar p_2 p_2|=\epsilon_2$ and choose $p_3\in B_{\epsilon_3}(\bar p_2)$, $\bar x_3\in B_\epsilon(x_3)$ so that $\geod{p_3\bar x_3}\subset \mathcal U$. Due to Lemma \ref{geod.close}, take $\epsilon_3\ll\epsilon_2$ small so that $\geod{y\bar x_2}_{X}$ are $\kappa$-geodesics for all $y\in B_{10\epsilon_3}(\bar p_2)$. In particular, $\geod{p_3\bar x_2}_{X}$ is a $\kappa$-geodesic. Take $\epsilon_3$ even smaller so that $\epsilon_3+2\epsilon_2<10\epsilon_2$. Thus $p_3\in B_{\epsilon_3}(\bar p_2)\subset B_{10\epsilon_2}(\bar p_1)$. Then we get $\geod{p_3\bar x_1}$ is a $\kappa$-geodesic.

  Select $\bar p_3, p_4,\bar p_4, \dots,p_N$ recursively with
  $$|p_j\bar p_i| =\epsilon_j+2\epsilon_{j-1}+\dots+2\epsilon_{i+2}+2\epsilon_{i+1}<10\epsilon_i,$$
  for any $j>i\ge 1$. Finally, we have that $\geod{p_N\bar x_i}$ are $\kappa$-geodesics for all $0\le i\le N$. The point $\bar p=p_N$ is the desired point since
  $$|\bar pp| =|p_Np|
  \le |p_N\bar p_1|+2\epsilon_1
  \le\epsilon_N+2\epsilon_{N-1}+\dots+2\epsilon_2+2\epsilon_1<\epsilon.$$
\end{proof}

\begin{proof}[{\bf Proof of Theorem B}]
  Assume $0<\lambda<1$. By Corollary \ref{comp.X.cor2} and Lemma \ref{pert.triangle} we get that $\bar{\mathcal U}\in\Alex(\kappa_1)$, where $\kappa_1=\lambda^2\kappa+(1-\lambda^2)\kappa_0$. Repeat this argument recursively with $\kappa_{i-1}$ being replaced by $\kappa_i$ for $i=1,2,\dots$. We get $\bar{\mathcal U}\in\Alex(\kappa_i)$, where $\kappa_{i+1}=\lambda^2\kappa+(1-\lambda^2)\kappa_i$. Clearly, $\kappa_i\to\kappa$ as $i\to \infty$. Thus $\bar{\mathcal U}\in\Alex(\kappa)$.
\end{proof}

\begin{proof}[{\bf Proof of Corollary \ref{cor.ThmA}}]
  We will show that almost everywhere convexity implies weak $\mathfrak p_1$-convexity if $\mathcal U\in\Alex_{loc}(\kappa)$ is finite dimensional.

  Let $p,q,s\in \mathcal U$. By the assumption, for any $r>0$, there is $p_1\in B_\epsilon(p)$, so that $\mathcal H^n\left(\mathcal U\setminus\cnnt{p_1}{\mathcal U}\right)=0$. Choose $q_1\in B_\epsilon(q)$ and $\bar s\in B_{\epsilon/10}(s)$ so that $\geod{q_1\bar s}\subset\mathcal U$. Take $s_0\in\geod{q_1\bar s}$ such that $|s_0\bar s|=\epsilon/10$. Let ${\epsilon_1}>0$ be small so that $B_{\epsilon_1}(s_0)\subset B_\epsilon(s)$. Suppose that there is $\delta>0$ such that for every $s_1\in B_{\epsilon_1}(s_0)$ and every geodesic $\geod{q_1s_1}$,
  \begin{align}
    \mathcal H^1\left(\geod{\bar q_1s_1}\setminus\cnnt{p_1}{\geod{q_1s_1}}\right)>0,
    \label{cor.ThmA.e1}
  \end{align}
  where $\bar q_1\in\geod{q_1s_1}$ such that $|q_1\bar q_1|=\delta$. The weak $\mathfrak p_1$-convexity follows if this is not true.

  By Lemma \ref{thin.comp.r} and Alexandrov's lemma, we can choose ${\epsilon_1}$ small so that for any $u\in B_{\epsilon_1}(s_0)$, $x\in\geod{q_1s_0}$ and $y\in\geod{q_1u}$, the comparison angle $\tang{\kappa}{q_1}{x}{y}$ is decreasing in $|q_1x|$ and $|q_1y|$. Thus we have
  \begin{align}
    \frac{|xy|}{|us_0|}\ge c(\kappa)\cdot\frac{\min\{|q_1x|, |q_1y|\}}{|q_1s_0|}.
    \label{cor.ThmA.e2}
  \end{align}
  Let
  $$T=\left(\bigcup_{s_1\in B_{\epsilon_1}(s_0)}\left(\geod{ q_1s_1}\setminus\cnnt{p_1}{\geod{q_1s_1}}\right)\right) \setminus B_\delta(q_1)$$
  and $S=\{v\in\mathcal U: |q_1v|=|q_1s_0|-{\epsilon_1}\}$. For any $x\in T$, $x\neq q_1$, there is $u_x\in B_{\epsilon_1}(s_0)$ such that $x\in T\cap\geod{q_1u_x}$. Let $\bar x\in\geod{q_1u_x}$ such that $|q_1\bar x|=|q_1s_0|-{\epsilon_1}$. Then $\bar x\in S$. Define a map $\phi\colon T\to S\times\mathbb R^+; x\to (\bar x, |q_1x|)$. Due to (\ref{cor.ThmA.e1}), we have $\mathcal H^1\left(\phi(T\cap\geod{q_1u_x})\right)>0$. Suppose that $\dim(\mathcal U)=n$. By Fubini's theorem,  $\mathcal H^n(\phi(T))>0$. Because $\dist{q_1,T}\ge\delta$, by (\ref{cor.ThmA.e2}), we see that $\phi$ is $c(\kappa,|q_1s_0|,\delta)$-co-Lipschitz. Therefore, $\mathcal H^n(T)>0$. On the other hand, $T\subset\mathcal U\setminus\cnnt{p_1}{\mathcal U}$. Thus $\mathcal H^n(T)\le\mathcal H^n(\mathcal U\setminus\cnnt{p_1}{\mathcal U})=0$, a contradiction.
\end{proof}

\section{Examples}

\begin{example}\label{example.1}
  Let $\mathcal U$ be an open metric ball with radius $r$ in the unit sphere. If $r\le\pi/2$, both $\mathcal U$ and $\bar{\mathcal U}$ are convex. Petrunin's Theorem \cite{Pet13} applies to this case. When $\pi/2<r<\pi$, neither $\mathcal U$ or $\bar{\mathcal U}$ is convex. When $r=\pi$, $\mathcal U$ is a.e. convex and Theorem A applies.
\end{example}

\begin{example}\label{eg.rational}
For any $\delta\in(0,1)$, we construct a locally flat, $2$-dimensional length space $\mathcal U$, whose metric completion $\bar{\mathcal U}\in\Alex(0)$, but $\mathcal H^2(\mathcal U)\le\delta\cdot \mathcal H^2(\bar{\mathcal U})$.

Let $X=[0,1]\times[0,1]$ be the unit square equipped with the induced Euclidean metric. Let $\{\gamma_i\}_{i=1}^\infty$ be the collection of all segments in $X$ whose coordinates of end points are pairs of rational numbers. Let $\mathcal U=\bigcup_{i=1}^\infty B_{r_i}(\gamma_i)$, where $\sum_{i=1}^\infty r_i=\delta/4$. Then $\mathcal H^2(\mathcal U)\le\delta< 1=\mathcal H^2(X)$. Now we show that $\bar{\mathcal U}=X$. For any two points $p, q\in X$ and any $\epsilon>0$, there are rational points $\bar p=(x_1,y_1)\in B_\epsilon(p)$ and $\bar q=(x_2,y_2)\in B_\epsilon(q)$. Thus for some $i$, $\geod{\bar p\bar q}=\gamma_i\subset \mathcal U$. We have
$$
  \textsf{d}_{\bar{\mathcal U}}(p, q)
  \ge\textsf{d}_X(p, q)
  \ge L(\gamma_i)-2\epsilon
  =\textsf{d}_{\mathcal U}(\bar p, \bar q)-2\epsilon
  \ge \textsf{d}_{\bar{\mathcal U}}(p, q)-4\epsilon.
$$
Therefore, $\bar{\mathcal U}$ is isometric to $X$.
\end{example}


%

\vskip 30mm

\bibliographystyle{amsalpha}

\end{document}